\documentclass[12pt,reqno]{amsart}
\usepackage{a4wide}
\usepackage{amssymb,amsmath,amsthm,newlfont,enumerate,color}
\newtheorem{theor}{Theorem }
\newtheorem{prop}{Proposition}[section]
\newtheorem{theo}[prop]{Theorem}
\newtheorem{lem}[prop]{Lemma}
 \newtheorem{coro}[prop]{Corollary}

\newcommand{\cM}{{\mathcal{M}}}

\newcommand{\un}{\mathbf{1}}

\newcommand\TT{\mathbb{T}}
\newcommand\DD{\mathbb{D}}

\newcommand\RR{\mathbb{R}}

\DeclareMathOperator{\supp}{supp}
\DeclareMathOperator{\Span}{span}

\newcommand{\cZ}{{\mathcal{Z}}}
\newcommand{\cJ}{{\mathcal{J}}}
\newcommand{\cI}{{\mathcal{I}}}
\newcommand{\cD}{{\mathcal{D}}}

\newcommand{\cR}{{\mathcal{R}}}
\newcommand{\cH}{{\mathcal{H}}}

\DeclareMathOperator{\dist}{dist}

\title[ Kernel and capacity in Dirichlet type spaces]
{Kernel estimate and  capacity in Dirichlet type spaces}

\author[O. El-Fallah,  Y. ElMadani, K. Kellay]
{O. El-Fallah,  Y. ElMadani, K. Kellay}
\address{O.El-Fallah \& Y. ElMadani, Laboratoire Analyse et Applications URAC/03\\ Universit\'e Mohamed V Agdal-Rabat- \\  B.P. 1014 Rabat\\Morocco}
\email{elfallah@fsr.ac.ma}
\email{madani@fsr.ac.ma}

\address{K. Kellay,  IMB\\Universit\'e de  Bordeaux \\
351 cours de la Lib\'eration\\33405 Talence \\France}
\email{kkellay@math.u-bordeaux1.fr}
\thanks{Research partially supported by "Hassan II Academy of Science and Technology" for the first and the second authors. }
\begin{document}
\maketitle
\begin{abstract}
Let $\mu$ be a positive finite measure on the unit circle. The Dirichlet type space $\cD (\mu)$, associated to $\mu$, consists of holomorphic functions on the unit disc whose derivatives are square integrable when weighted against the Poisson integral of $\mu$. First,  we give an estimate of the norm of the reproducing kernel $k^\mu$ of $\cD (\mu)$. Next, we study the notion of $\mu$-capacity associated to $\cD (\mu)$, in the sense of Beurling--Deny.  Namely, we give an estimate of  $\mu$-capacity of arcs in terms of  the norm of $k^\mu$. We also provide a new condition on closed sets to be $\mu$-polar. Note that in the particular case where $\mu$ is the Lebesgue measure,  this condition coincides with  Carleson's condition \cite{Ca}. Our method is based on  sharp estimates of norms of some outer test functions which allow us to transfer these problems to an estimate of the reproducing kernel of an appropriate weighted Sobolev space.

\end{abstract}

\section{Introduction}

Let $H^2$  denote the classical {\it Hardy space} of analytic functions on the unit disc $\DD$ having square summable Taylor coefficients at the origin. Every function $f\in  H^2$ has non-tangential limits almost everywhere on the unit circle $\TT=\partial \DD$. We denote by $f(\zeta)$ the non-tangential limit of $f$ at $\zeta\in \TT$ if it exists. \\
Let  $\mu$  be a positive  finite  measure on $\TT$,  the {\it Dirichlet type   space} $  \cD(\mu)$  is the set of analytic  functions $f\in H^2$, such that
 $$\cD_{\mu}(f):=\frac{1}{2\pi}\int_{\TT}\int_{\TT}\frac{|f(\zeta)-f(\xi)|^2}{|\zeta-\xi|^{2}}|d\zeta|d\mu(\xi)<\infty,$$
 The space $\mathcal{D}(\mu)$ is endowed with the  norm 
$$\|f\|_{\mu}^{2}:=\|f\|^{2}_{H^2}+\cD_{\mu}(f).$$
If  $d\mu(e^{it}) = 0$, then $\cD( \mu)=H^2$ and if $d\mu(e^{it})=dm(t)=d t/2\pi$, the normalized arc measure on $ \TT$, then $  \cD( \mu)$ is the classical {\it Dirichlet space} $\cD$.\\
These spaces were introduced by Richter  by examining the {\it $2$-isometries} on the Hilbert spaces. 
A bounded operator $T$ in a Hilbert space $\cH$ is called {\it $2$-isometry} if ${T^*}^{2}T^2-2T^*T-I=0$, is said to be {\it cyclic} if there exists $x\in \cH$ such that $\Span\{T^n x,n \geq 0\}$  is dense in $\cH$ and is  called {\it analytic} if $\bigcap_{n\geq0}T^n\cH=\{0\}$.  In \cite{Ri}, Richter proved that every cyclic, analytic  $2$--isometry can be represented as a multiplication by $z$ on a Dirichlet type space    $\cD( \mu)$ for some positive    finite  measure $\mu$. 
As consequence \cite{Ri, Ri1} Richter gave  an analogue of Beurling's theorem for the Dirichlet space.

\subsection{Reproducing kernels} The {\it reproducing kernel} plays an important role in the study of Hilbert
spaces of analytic functions. In particular, it allows to determine the rate of growth of functions
near the boundary and its tangential behavior; their properties are closely related to embedding theorems, sampling and interpolation sets, and other topics.  Let 
 $P[\mu]$ be the Poisson integral of the positive  finite measure $\mu$ on $\TT$
$$
P[\mu](z)=\int_\TT\frac{1-|z|^2}{|\zeta-z|^2}d\mu(\zeta),\qquad z\in \DD.
$$

In the following theorem, we provide an asymptotic estimate 
of the reproducing kernel   $k^\mu$ of $  \cD(\mu)$ on the diagonal. 

\begin{theor}\label{thkernel} Let $\mu$ be a finite positive measure on $\TT$. We have
  $$  k^\mu(z,z)\asymp 1+ \int_{0}^{|z|} \frac{dr}{(1-r)P[\mu](rz/|z|)+(1-r)^2},$$
 where the implied constants are absolute.
  \end{theor}

 Let us recall that Shimorin \cite{S1}  proved that  all Dirichlet type  spaces have complete Nevanlinna--Pick reproducing kernels.  As an important consequence (see \cite[\S 9.4]{AM} and \cite[Theorem 1]{Seip}), each  sequence $\cZ=\{z_n\}\subset \DD$ satisfying Shapiro--Shields condition $\sum_{z\in \cZ}1/k^\mu(z,z)<\infty$
 is a zero set of $\cD(\mu)$. Theorem  \ref{thkernel}  allows to give  examples of zero sets of $\cD (\mu )$.\\

To prove the lower estimates of $ k^\mu(z,z)$, we establish a sharp norm estimates of some outer functions which peak near $z$. This allows us to transfer our problem to an estimation of the norm of the kernel of an appropriate weighted Sobolev space. In fact and roughly speaking, Theorem \ref{thkernel} says that $ k^\mu(z,z)\asymp K_{\varphi}(1-|z|,1-|z|)$ where $K_{\varphi}$ is the reproducing kernel of the weighted Sobolev space defined by
$$
{W^2(\varphi):= \big\{ f\in C \big ( (0, 2\pi] \big )\text{ : }  f(x)= f(1)+\displaystyle \int _{x}^{2\pi} g(t)dt , \
 g \in L^2((0,2\pi), \varphi dt) \big\}},
$$
where $\varphi (t) = tP[\mu]((1-t)z/|z|)+t^2$.

\subsection{Capacity} The Dirichlet type  space is closely related to  some notions of potential theory.
Let $\cD^h(\mu)$ be the harmonic version of $  \cD(\mu)$ given by
$$\cD^h(\mu):=\left\{f\in L^2(\TT) \text{ : } \|f\|_{\mu}^{2}:=\|f\|_{L^{2}(\TT)}^{2}+\cD_\mu(f)<\infty\right\}.$$
 $\cD^h(\mu)$ is a Dirichlet space in the sense of Beurling--Deny \cite{BD}. Some aspects of the potential theory associated to the general Dirichlet spaces were studied in several papers (see for instance \cite{FOT}).  In this paper we will focus on the  notion of capacity. We recall at first, the definition of capacity in the sense of Beurling--Deny.  Let $U$ be a open subset of the unit circle. The $c_{\mu}$-capacity of $U$ is defined by
\begin{equation}
\label{cmucap}
c_\mu (U):=\inf\left\{\|u\|_\mu^2 \text{ : } u\in \cD^h (\mu),  \; \; u\geq 0 \text{ and } u \geq 1 \text{ a.e. on } U \right\}.
\end{equation}
 As usual we define the $c_{\mu}$-capacity of any subset $F \subset  \TT$  by 
 $$
 c_{\mu}(F)= \inf \{ c_{\mu}(U):  UÊ\; \text{ open,  }\;  F\subset U \}.
 $$
Since the $L^2$ norm dominates the Dirichlet--type norm, it is completely obvious that sets having 
$c_\mu$--capacity $0$ have Lebesgue measure $0$. We say that a  property  holds $c_\mu $-quasi-everywhere ($c_\mu$--q.e.)  if it holds everywhere outside a set of $c_\mu$--capacity $0$.  So,  $c_\mu$--q.e. implies a.e.. A closed set of capacity zero will be called, throughout this paper, {\it $\mu$--polar} set. If  $d\mu(e^{it}) = d t/2\pi$, the normalized arc measure on $ \TT$, then $  \cD(\mu)$ is the classical Dirichlet space $\cD$ and $c_\mu $ is comparable to the logarithmic capacity,  see \cite[Theorem 14]{M} and \cite[Theorem 2.5.5]{AH}.\\
 Our first result on $\mu$-capacity gives an estimate of capacity of arcs in terms of the kernel.  More precisely we have:
 \begin{theor}\label{cap arc} Let $I\subset \TT$ be the arc  of length $|I|=1-\rho$ with the midpoint at $\zeta\in \TT$. Then
$$c_\mu(I)\asymp \frac{1}{k^{\mu}(\rho\zeta ,\rho\zeta )}.$$
where the implied constants are absolute. 
\end{theor}
As a consequence,  
\begin{equation}\label{cappoint}
c_\mu(\{e^{i\theta}\})=0\iff  \int_{0}^{1} \frac{dr}{(1-r)P[\mu](re^{i\theta}) +(1-r)^2}=\infty.
\end{equation}
In the sequel we will suppose that  $E$ is a closed set which has Lebesgue measure zero and $\mu$ is a finite positive measure on $\TT$. Now our goal is to  give sufficient condition on $E$ to be $\mu$--polar. Let us introduce the local modulus of continuity of $\mu$ on $E$ which will play a crucial role in this paper. It is defined by 
\begin{equation}\label{locmod}
\rho_{\mu,E} (t):= \sup \{\mu (\zeta e^{-it},\zeta e^{it})\text{ : } \zeta \in E\}.
\end{equation}
Note that $\rho_{\mu,\TT} =\rho_{\mu}$ is the classical modulus of continuity of $\mu$.
Let us write 
$$E_t:=\{\zeta\in \TT\text{Ê: } d(\zeta,E)\leq t\},$$ where $d$ denotes the distance with respect to arc-length, and denote by $|E_t|$  the Lebesgue measure of $E_t$.  
We can express the function $|E_t|$ in terms of 
$${N}_E(t):=2\sum_j\un_{\{|I_j|>2t\}},$$
where $(I_j)$ are the components of $\TT\backslash E$, as follows
$$\int_0^t {N}_E(s)ds=|E_t|.$$

 In Theorem \ref{capdirichlet}, we give a sufficient conditions on a closed subset $E$, in terms of $\rho_{\mu,E} $ and ${N}_E$, to be $\mu$-polar.
To illustrate this theorem we give here some of its corollaries.\\

\noindent{\bf(i)} If 
$$\displaystyle \int_0^\pi \frac{dt}{ \int_0^t ( \rho_\mu(s)N_E(s)/s)ds}=+\infty, $$
then $c_\mu(E)=0$.\\

This result can be considered as an extension of Carleson's Theorem \cite[section IV, Theorem 2]{Ca}. In fact if $\mu =m$ is the Lebesgue measure then $\rho_{\mu,E}(t)= t$ and  $c_\mu \asymp c$ ($c$ is the logarithmic capacity). We obtain  Carleson's theorem which says that if 
$\int _0^\pi{dt}/{|E_t|}=\infty,$ then $c(E)=0$.\\

\noindent{\bf(ii)} Suppose that $\rho_{\mu,E}(t)=O(t^\alpha)$ with  $1\leq \alpha <2$. If 
 $$\displaystyle \int_0^\pi \frac{dt}{t^{\alpha-1}|E_t|}=+\infty, $$
 then
  $c_\mu(E)=0$.\\

\noindent{\bf(iii)} If $ \rho_{\mu,E}(t)=O(t^\alpha)$ with  $\alpha>2$, we have  $c_\mu(E)=0$.\\

\noindent Note also that if $t^\alpha =O(\rho_{\mu,\{1\}}(t))$ with  $ \alpha <1$, then by \eqref{cappoint},   $c_{\mu} (\{1\})>0$.\\

The proof of Theorem  \ref{capdirichlet} uses an  idea analogous to the   proof of Theorem \ref{thkernel}. However, our test functions must peak on the whole set $E$ and the desired weighted Sobolev spaces will depend  on $\mu$ and $E$. In fact, we prove that
there is no bounded point evaluation at $0$  for $W^2(\varphi)$ (where $\varphi$ depends on $\rho _{\mu , E}$ and $N_E$), then $c_\mu(E) =0$. Note finally, that  there is no bounded point evaluation at $0$  for $W^2(\varphi)$  if and only if $\displaystyle \lim _{t\to 0^+}K_{\varphi}(t,t)= \infty$.\\

The plan of the paper is the following. In the next section we recall two formulas of the Dirichlet type  norm;  we also give punctual estimates of some outer functions.  In Section 3 we give norm estimates of our test functions. In Section 4 we give  diagonal asymptotic estimates of reproducing kernel.  In section 5, we prove the announced results on capacity.\\

Throughout the paper, we use the following notations:
\begin{itemize}
\item $A\lesssim B$ means that there is an absolute constant $C$ such that $A \leq CB$. 
\item $A\asymp B$  if both $A\lesssim B$ and $B\lesssim A$. 
\item $C(x_1,\ldots,x_n)$ denote a constant which depends only on variables $x_1, \ldots, x_n$. 
\end{itemize}

\section{Preliminaries}

\subsection{Norm formulas} 
In this subsection we recall some results  about norm formulas in Dirichlet type  spaces which will be used in what follows. \\

For a finite positive measure $\mu$ on $\TT$, the harmonic Dirichlet space $\cD^h(\mu)$ consists of functions $f\in L^2(\TT)$ such that 
$$\cD_{\mu}(f):=\int_{\TT}\cD_\xi(f)d\mu(\xi)<\infty,$$
 where $\cD_\xi(f)$ is the {\it local Dirichlet integral} of $f$ at $\xi\in \TT$ given by
 $$\cD_{\xi}(f):=\int_{\TT}\frac{|f(e^{it})-f(\xi)|^2}{|e^{it}-\xi|^{2}}\frac{dt}{2\pi}.$$

 The Douglas' formula, see \cite[Theorem 7.1.3]{EKMR}, expresses the Dirichlet integral of a function $f$ in terms of the Poisson transform of $\mu$, namely
$$D_\mu(f)=\int_\DD |\nabla P[f] |^2P[\mu] dA,\qquad f\in D^h(\mu),$$
where $dA(z)=dxdy/\pi$ stands for the normalized area measure in $\DD$. 
In particular, if $f\in \cD(\mu)( =\cD^h(\mu)\cap H^2)$,  Douglas' Formula becomes 
  $$\cD_{\mu}(f):=\int_\DD |f'(z)|^2 P[\mu](z) dA(z)<\infty. $$

Another useful formula, due to Richter and Sundberg  \cite[Theorem 3.1]{RS1}, gives the local Dirichlet integral of function $f$ in terms of their zeros sequence, their singular measure and the modulus of their radial limit. We will need,  throughout this paper, the Richter--Sundberg formula mainly for outer functions.  
Recall that outer functions are given by  
$$f(z)=\exp \int_\TT\frac{\zeta+z}{\zeta-z}\log \varphi(\zeta) \frac{|d\zeta|}{2\pi}, \qquad (z\in \DD),$$
where $\varphi$ is a positive function such that $\log \varphi\in L^1(\TT)$.  Note that the radial limit of $f$, noted also by $f$, exists a.e. and $|f|=\varphi$  a.e. on $\TT$.\\

Let $f\in H^2$ be an outer function such that $f(\zeta)$ exists at $\zeta\in \TT$. We have
\begin{equation}\label{formuleRS}
\cD_\zeta(f)=\int_\TT\frac{|f(\lambda)|^2-|f(\zeta)|^2-2|f(\zeta)|^2\log|f(\lambda)/f(\zeta)|}{|\lambda-\zeta|^2}\frac{|d\lambda|}{2\pi}.
\end{equation}

 \subsection{{Punctual estimates of test functions}} The result obtained in this subsection will be used in the proof of the lower estimate of the kernel.
\begin{lem}\label{meshar}Let $1/2<r=1-a<1$  and let  $I_k=[e^{ia_k},e^{ia_{k+1}}]$ with  $a_0=0$, $a_k=2^{k}a$ ($k\geq 1$). Let $N$ be the integer such that $2^{N}a\leq \pi < 2^{N+1}a$, then 
$$\sum_{k=0}^{N-1}(k+1)\varpi(r,I_k,\DD) \asymp 1,$$
where $\varpi(r,I_k,\DD)$ denotes the harmonic measure of $I_k$ at $r$ in $\DD$.
\end{lem}

\begin{proof}
Without loss of generality, we may suppose that $2^Na= \pi$. Note that 
$$
\sum_{k=0}^{N-1}(k+1)\varpi(r,I_k,\DD) \geq \sum_{k=0}^{N-1}\varpi(r,I_k,\DD) \asymp 1.
$$
For the reverse inequality, let $g(z)=\log1/|1-rz|$. Since  $g$ is harmonic in the   neighbourhood of $\overline{\DD}$, 
$$g(z)=\frac{1}{2\pi}\int_{-\pi}^{\pi}\frac{1-|z|^2}{|1-ze^{-i\theta}|^2}\log \frac{1}{|1-re^{i\theta}|}d\theta.$$
So,
\begin{equation}\label{har}
g(r)=\log\frac{1}{1-r^2}
=\sum_{k=0}^{N-1}\frac{1}{\pi}\int_{a_k}^{a_{k+1}} \frac{1-r^2}{|1-re^{-i\theta}|^2}\log \frac{1}{|1-re^{i\theta}|}d\theta.
\end{equation}
For $ k=0, \ldots,  N-1,$ and  $\theta\in (a_k,a_{k+1})$, we have 
$$\frac{1}{|1-re^{i\theta}|}\asymp \frac{1}{2^k}\frac{1}{1-r}.$$
By \eqref{har}, we get 
$$\log\frac{1}{1-r^2}= \log \frac{1}{1-r}-\log 2\sum_{k=0}^{N-1}k\varpi(r,I_k,\DD)+ O(1)$$
and our result follows.
\end{proof}

 Let  $w:(0,\pi)\to (0,+\infty)$ be a continuous  positive  function such that $\log w\in L^1(\TT)$.  As before  $f_w$ denote the outer function satisfying 
  \begin{equation}\label{outerfunction}
  |f_w(e^{it})|=w(|t|) \quad \text { a.e} \ \mbox{on}\  (-\pi,\pi).
  \end{equation}

\begin{prop}\label{estimationharmo}
Let $w:(-\pi,\pi)\to (0,+\infty)$ be an even continuous positive decreasing  function such that   $\log w\in L^1(\TT)$. Suppose that $w(x)\leq 2 w(2x)$.  Let  $f_w$ be an outer function given by 
\eqref{outerfunction}. Then 
$$w(1-r) \lesssim  |f_w(r)|, \qquad 0\leq r<1.$$
\end{prop}
\begin{proof}
Let $a$ , $I_k$  and $N$ as in lemma \ref{meshar}, and suppose that $a_N= \pi$. We have 
\begin{eqnarray*}
 |f_w(r)|&= &\exp \Big \{ \sum_{k=0}^{N-1} \frac{1}{\pi}\int_{I_k}\frac{1-r^2}{|1-re^{i\theta}|^2}\log w(\theta) d\theta \Big \} \\
&\geq & \exp \Big \{\sum_{k=0}^{N-1} \log \omega(2^{k+1} a) \frac{1}{\pi}\int_{I_k}\frac{1-r^2}{|1-re^{i\theta}|^2} 
d\theta \Big \}\\
&\geq  & \exp  \Big \{ \sum_{k=0}^{N-1} (\log \omega(a)-(k+1)\log 2) \frac{1}{\pi}\int_{I_k}\frac{1-r^2}{|1-re^{i\theta}|^2}d\theta  \Big \}\\
&\geq  &\exp \Big  \{ \log w (a) \sum_{k=0}^{N-1}\varpi(r,I_k,\DD) -\log 2 \sum_{k=0}^{N-1}(k+1)\varpi(r,I_k,\DD)  \Big \}.  \\
\end{eqnarray*}
We obtain from Lemma \ref{meshar}, that $w(1-r) \lesssim |f_w(r)|$. The case $a_N<\pi$ can be treated in the same way by taking into account the interval  $[e^{ia_N},e^{i\pi}]$.
\end{proof}

\subsection{Regularization lemma}Let $\mu$ be a positive finite measure on $\TT$,  we set $d\mu(s)=d\mu(e^{is})$. Denote by  
$$\widehat{\mu} (s)=\mu([e^{-is},e^{is}])\ \ (0\leq s \leq \pi)\ \ \mbox{and}\ \ \widehat{\mu} (s) = \widehat{\mu} (\pi)\ \ (s>\pi).$$ 
Let
 \begin{equation}\label{pfmu}
 F_\mu(x)=\int_{-\pi}^{\pi}\frac{x^2}{x^2+s^2}d\mu(s)\qquad x>0.
 \end{equation}
Note that $F_\mu$ is increasing and $F_\mu(x)/x^2$ is decreasing. We extend  $F_\mu$ at the origin by $ F_\mu (0)= F_\mu (0^+)$. In the following lemma we collect some elementary properties of $F_\mu$ which will be used in the sequel.
\begin{lem}\label{lemFmu1}  Let $\nu$ be  a positive finite measure on $\TT$.  We have  the following
 \begin{enumerate}
\item $ F_\nu(x)\asymp x P[\nu](1-x)$, for $x>0$,\\
\item $ \widehat{\nu} (x)\lesssim  F_\nu(x)$  for $x\geq 0$, \\
\item $\displaystyle  \int_{x\leq |s|\leq \pi}\frac{d{\nu} (s)}{s^2}\lesssim  \frac{F_\nu(x)}{x^2}$,  for $x>0$,\\
\item If $h$ is a positive monotone function on $(0,\pi)$. Then 
$$\int_{-a}^{a}h(|x|)d\nu(x)\lesssim \int^{a}_{0}\frac{h(x)}{x}\widehat{\nu}(2x)dx,$$
\item If $\widehat{\nu}(2x)\leq c \widehat{\nu}(x)$ for some constant $c<4$,  then $\displaystyle F_\nu(x)\lesssim \frac{\widehat{\nu} (x)}{4-c}$.
\end{enumerate}
\end{lem}

\begin{proof}  (1), (2) and (3) are obvious.
To prove (4) suppose that $h$ is a decreasing function. Clearly if $\nu (\{ 0\})>0$, then (4) is obvious. So, suppose that  $\nu (\{ 0\})=0$. We have
\[
\begin{array}{ccl}
\displaystyle \int _{-a}^{a}h(|x|)d\nu(x) & =  &\displaystyle  \sum_{n\geq 0}  \int_{2^{-(n+1)}a}^{ 2^{-n}a}h(x)  (d\nu(x)+d\nu(-x)) \\

& \lesssim &\displaystyle  \sum_{n=0}^{m} h(2^{-n}a)\widehat{\nu}(2^{-n}a)\\
&\lesssim &  \displaystyle  \int^{a}_{0}\frac{h(x)}{x}\widehat{\nu}(2x)dx.\\

\end{array}
\]
Analogue argument works if $h$ is increasing.\\
Finally to prove (5), suppose that $\widehat{\nu}(2x)\leq c \widehat{\nu}(x)$  with $c<4$. We have $ \widehat{\nu}(2^nx)\leq c^n\widehat{\nu}(x)$ and 
$$
\int_{|t|\geq {x}} \frac{d\nu(t)}
{t^2}dt=\sum_{n\geq 0}\int\limits_{2^nx\leq |t|\leq 2^{n+1}x }\frac{d\nu(t)}{t^2}
\leq\sum_{n\geq 0}\frac{\widehat{\nu} (2^{n+1}x)}{2^{2n}x^2}
\leq\frac{4c}{4-c}\frac{\widehat{\nu} (x)}{x^2}. 
$$
So 
$$F_\nu(x)\leq \widehat{\nu} (x)+  \int_{|\theta|>x}\frac{x^2}{x^2+\theta^2}d\nu(\theta)\leq \widehat{\nu} (x)+x^2\int_{\theta\geq x}\frac{d\nu(x)}{\theta^2}\lesssim \frac{ \widehat{\nu} (x)}{4-c} .$$

\end{proof}


\section{Norm estimate of test functions}

\subsection{Norm estimate of  analytic test functions }The purpose of this subsection is to give estimate of norms of some outer functions which play an important role in what follows.

 The following lemma is the first step to prove Theorem \ref {coronorm}.
 
 \begin{lem} \label{normf} Let  $w:[0, \pi]\to (0,+\infty)$ be a $C^1$  decreasing  convex function such that $w(x)\leq 2w(2x)$. Suppose that $ x^2|w'(x)|$ is increasing and let $f_w$ be the outer function given by \eqref{outerfunction}. Then 
 $$ \cD_\mu(f_w)\lesssim \cJ_1+\cJ_2+\cJ_3,$$
where 
\begin{align*}
 \cJ_1&:=\int_{x=0}^{\pi}\int_{y=0}^{x}  |w'(y)|w(y)\frac{|w'(x)|}{w(x)}  \frac{\widehat{\mu}(y)}{x} dxdy, \\
 \cJ_2 &:=\int_{s=-\pi}^{\pi}  w'(s)^2sd\mu(s),\\
 \cJ_3 &:= \int_{x=0}^{\pi}\int_{y=x}^{\pi}x  |w'(y)|w(y)\frac{|w'(x)|}{w(x)}
\Big( \int_{y\leq |s|\leq \pi}\frac{d\mu(s)}{s^2}\Big) dxdy.  
\end{align*}
 \end{lem}
 \begin{proof} Without loss of generality, we may assume that $d\mu(s)=d\mu(-s)$. By Richter-Sundberg formula \eqref{formuleRS}
 we have 
 $$\cD_{e^{is}}(f_w)=\frac{8}{2\pi}\int_{t=0}^{\pi}\int_{x=s}^{t}\int_{y=s}^{x}w'(y)w(y)\frac{w'(x)}{w(x)}dydx\frac{dt}{|e^{is}-e^{it}|^2}.$$
 So, 
 \begin{eqnarray*}
 \cD_\mu(f_w)&=&\frac{16}{2\pi}\int_{s=0}^{\pi} \int_{t=0}^{\pi} \int_{x=s}^{t}\int_{y=s}^{x}w'(y)w(y)\frac{w'(x)}{w(x)}dydxdt \frac{d\mu(s)}{|e^{is}-e^{it}|^2}\\
 &=&\underbrace{\int_{s=0}^{\pi}\int_{t=2s}^{\pi}\ldots }_{\cI_1}+\underbrace{\int_{s=0}^{\pi}\int_{t=s/2}^{2s}\ldots}_{\cI_2}+
\underbrace{\int_{s=0}^{\pi}\int_{t=0}^{s/2}\ldots}_{\cI_3}.
\end{eqnarray*}
To complete the proof we will estimate each term separately. \\ 

If $2s\leq t\leq \pi$, we have  $|t-s|\geq t/2$,  
\begin{eqnarray*}
\cI_1&\lesssim &\int_{s=0}^{\pi}\int_{t=s}^{\pi}\int_{x=s}^{t} \int_{y=s}^{x} |w'(y)|w(y)\frac{|w'(x)|}{w(x)}dydx\frac{dt}{t^2}d\mu(s)\\
&\lesssim&\int_{t=0}^{\pi}\int_{x=0}^{t}\int_{y=0}^{x}|w'(y)|w(y)\frac{|w'(x)|}{w(x)} \Big(\int_{|s|\leq y}d\mu(s)\Big)\frac{dt}{t^2}dydx \\
&\lesssim&\int_{x=0}^{\pi}\int_{y=0}^{x}|w'(y)|w(y)\frac{|w'(x)|}{w(x)} \frac{\widehat{\mu}(y)}{x}dydx
\end{eqnarray*}

Since $w(2x) \asymp w(x)$ and $w'(2x)\asymp w'(x)$, we have
 $$\cI_2\lesssim \int_{s=-\pi}^{\pi}  w'(s)^2sd\mu(s).$$

If $0\leq t\leq s/2$, then  $|t-s|\geq s/2$,  
\begin{eqnarray*}
\cI_3&\lesssim & \int_{s=0}^{\pi}\int_{t=0}^{s}\int_{x=t}^{s}\int_{y=x}^{s}|w'(y)|w(y)\frac{|w'(x)|}{w(x)}dydx \frac{d\mu(s)}{s^2}dt\\
&\lesssim&\int_{t=0}^{\pi}\int_{x=t}^{\pi}\int_{y=x}^{\pi}|w'(y)|w(y)\frac{|w'(x)|}{w(x)} \Big(\int_{y\leq |s|\leq \pi}\frac{d\mu(s)}{s^2}\Big){dt} \\
 &\lesssim & \int_{x=0}^{\pi}\int_{y=x}^{\pi}|w'(y)|w(y)\frac{|w'(x)|}{w(x)} 
 \Big(\int_{y\leq |s|\leq \pi }^{\pi}\frac{d\mu(s)}{s^2} \Big)xdxdy.
\end{eqnarray*}
\end{proof}
\begin{theo}\label{coronorm}Let  $w:[0, \pi]\to (0,+\infty)$ be a $C^1$  decreasing  convex function such that $w(x)\leq 2w(2x)$. Suppose that $ x^2|w'(x)|$ is increasing and let $f_w$ be the outer function given by \eqref{outerfunction}. 
Then 
  $$ \cD_\mu(f_w)\lesssim  \|F_\mu w'   \|_\infty \|w\|_\infty,$$
where $F_\mu$ is given by \eqref{pfmu}.
 \end{theo}
\begin{proof} 
By (3) of Lemma \ref{lemFmu1}, we have  $\widehat{\mu}\leq F_\mu$. Now  Lemma \ref{normf} gives  
  \begin{eqnarray*}
\cJ_3&\lesssim &
\int_{x=0}^{\pi}\int_{y=x}^{\pi}|w'(y)|w(y)\frac{|w'(x)|}{w(x)}\frac{F_\mu(y)}{y^2} x dydx   \\
& \lesssim &   \|F_\mu w'  \|_\infty\int_{x=0}^{\pi}\int_{y=x}^{\pi} |w'(x)|\frac{w(y)x}{w(x)y^2} dydx 
\\
& \lesssim&  \|F_\mu w'  \|_\infty \int_{x=0}^{\pi}\int_{y=x}^{\pi} |w'(x)|\frac{x}{y^2} dydx \\
 &\lesssim&    \|F_\mu w'  \|_\infty \| w\|_\infty.
    \end{eqnarray*}
Note that $ x|w'(x)| \lesssim w(x)$ for all $x\in ]0,\pi]$. Indeed, since $w(2x)\asymp w(x)$ and $|w'(2x)|\asymp |w'(x)|$, it suffices to prove the inequality for $x\in [0,\pi /2]$. We have
 
$$w(x)\geq \int_x^{\pi} {t^2}|w'(t)|\frac{dt}{t^2}\geq x^2|w'(x)|\Big(\frac{1}{x}-\frac{1}{\pi}\Big)\geq \frac{x}{2} |w'(x)|.$$ 
So, again by Lemma \ref{lemFmu1} we get 
\begin{eqnarray*}
\cJ_1
 & \lesssim  &  \|F_\mu w'  \|_\infty\int_{x=0}^{\pi}\int_{y=0}^{x}|w'(x)|\frac{w(y)}{w(x)x} dydx\\
 &\lesssim&  \|F_\mu  w'  \|_\infty \int_{x=0}^{\pi}\int_{y=0}^{x} \displaystyle \int _{u=y}^\pi \frac{|w'(x)||w'(u)|}{w(x)x} dudydx\\
 &+&   \|F_\mu w'  \|_\infty\int_{x=0}^{\pi}\int_{y=0}^{x}|w'(x)|\frac{w(\pi)}{w(x)x} dydx\\
 &=& \cJ_{12}+\cJ_{22}.
 \end{eqnarray*}
 We have 
\begin{eqnarray*}
 \cJ_{12}&\lesssim&  \|F_\mu  w'  \|_\infty \int_{u=0}^{\pi}\int_{x=0}^{\pi}\int_{y=0}^{\min(u, x)} \frac{|w'(x)||w'(u)|}{w(x)x} dudydx\\
 &\lesssim& \|F_{\mu}  w' \|_\infty \Big( \int_{u=0}^{\pi}\int_{x=0}^{u} \frac{|w'(x)w'(u)|}{w(x)} dudx + \int_{u=0}^{\pi}\int_{x=u}^{\pi}\frac{u|w'(x)w'(u)|}{xw(x)} dudx\Big)\\
 &\lesssim&\|F_{\mu}  w'  \|_\infty\Big(\int_{x=0}^{\pi} |w'(x)|dx+
\int_{u=0}^{\pi} \int_{x=u}^{\pi}  \frac{u|w'(x)w'(u)|}{xw(x)} dudx )\\
 &\lesssim& \|F_{\mu}  w'  \|_\infty \Big(\|w\|_\infty+
 \int_{u=0}^{\pi}\int_{x=u}^{\pi}\frac{u|w'(u)|}{x^2} dudx \Big)\\
  &\lesssim& \|F_\mu w'  \|_\infty\|w\|_\infty. 
 \end{eqnarray*}
 Since $w$ is decreasing, we get
 \begin{eqnarray*}
 \cJ_{22}&\leq &\|F_\mu w'  \|_\infty\int_{x=0}^{\pi}\int_{y=0}^{x}|w'(x)|\frac{dx}{x} dy\\
 & \leq &  \|F_\mu w'  \|_\infty \|w\|_\infty.\\
  \end{eqnarray*}
Finally, applying (4) of  Lemma \ref{lemFmu1} with  $d\nu(s)=sd\mu(s)$,  we have $\widehat{\nu}(t)\leq t\widehat{\mu}(t)$ and 
$$\cJ_2 \lesssim\int_{0}^{\pi} \frac{w'(t)^2}{t}\widehat{\nu}(2t)dt\lesssim \|F_\mu w'  \|_\infty \|w\|_\infty.$$
   \end{proof}

\subsection{Norm estimate of  test functions  in $\cD^h(\mu)$} Our goal here is to give an estimate of the norm of some distance functions in $\cD^h(\mu)$ (For analytic distance functions see \cite{EKR}).The result of this subsection will be used in the proof of Theorem \ref{capdirichlet}.\\
Let $E$ be a closed subset of $\TT$, $\mu$ be a positive finite measure and denote by  
$\rho_{\mu,E}$ the local modulus of continuity  of $\mu$ on $E$ given by \eqref{locmod}. Note that  $\rho_{\mu, \{1\}}(t)=\widehat{\mu}(t)$. Recall   that  ${N}_E(t):=2\sum_j\un_{\{|I_j|>2t\}}$.  
where $(I_j)$ are the components of $\TT\backslash E$.
\begin{lem} \label{lem11}Let $\Omega:(0,\pi]\to \RR^+$ be a positive decreasing  function,  then  
$$\int_\TT\Omega(\dist(\zeta,E))d\mu(\zeta)\lesssim \int_0^1\Omega(t)\frac{\rho_{\mu,E}(t)}{t}{N}_E(t) dt.$$
\end{lem}
\begin{proof} 
Write $\TT\setminus E=\cup_n I_n$, where $(I_n)_n=(e^{i\alpha_n},e^{i\beta_n})_n$ are the components of $\TT\setminus E$.   
Let $d\mu_n(t)=d\mu(t+\alpha_n)+d\mu(\beta _n-t)$, By Lemma \ref{lemFmu1}
\begin{align*}
\int_{I_{n}}\Omega(\dist(\zeta,E ))d\mu(\zeta)& \asymp \int_{0}^{|I_n|/2}\Omega(t)d\mu_n(t)\\
&\lesssim  \int_{0}^{|I_n|/2}\Omega(t)\frac{\widehat{\mu_n}(t)}{t}dt\\
&\lesssim  \int_{0}^{|I_n|/2}\Omega(t)\frac{\rho _{\mu,E}(t)}{t}dt.\\
\end{align*}
Summing over all $I_n$, we get 
$$\int_\TT\Omega(\dist(\zeta,E ))d\mu(\zeta)\lesssim \int_{0}^{1}\Omega(t)\frac{\rho_{\mu,E}(t)}{t}{N}_E(t)dt,
$$
and the proof is complete.
\end{proof}

\begin{lem}\label{lem22}Let $E$ be a closed subset of $\TT$. Let $w$ be a convex  decreasing function  and let  $\Omega(\zeta)=w(d(\zeta,E))$. Then 
$$\cD_\mu(\Omega) \lesssim \cI_1+\cI_2+\cI_3, $$
where 
\begin{align*}
\cI_1&:=\int_{0}^{\pi}\int_{0}^{\pi}\frac{(w(t)-w(t+s))^2}{s^2} \frac{\rho_{\mu,E}(t)}{t}{N}_E(t)dsdt,\\
\cI_2&=\int_0^{\pi} w'(t )^2\rho_{\mu,E}(2t)  N_E(t)dt,\\
\cI_3&:=\int_{t=0}^{\pi}\int_{s=t}^{\pi}\frac{(w(t)-w(t+s))^2}{s^2} \frac{\rho_{\mu,E}(s)}{s}{N}_E(t)dsdt,
\end{align*}
\end{lem}
\begin{proof} 
Set $\delta=d(\zeta,E)$ and $\delta'=d(\zeta',E)$. By Lemma \ref{lem11} we have
\begin{eqnarray*}
\cJ_1&=&\frac{1}{2\pi}\int_\TT\int_{\delta'\leq \delta}\frac{(w(\delta)-w(\delta'))^2}{|\zeta-\zeta'|^2} |d\zeta|d\mu(\zeta')\\
&\lesssim& \int_\TT\int_{\delta'\leq \delta}\frac{(w(\delta')-w(\delta'+|\zeta-\zeta'|))^2}{|\zeta-\zeta'|^2}|d\zeta| d\mu(\zeta')\\
&\lesssim&\int_\TT\int_{s=0}^{\pi}\frac{(w(\delta')-w(\delta'+s))^2}{s^2} ds d\mu(\zeta') \\
&\lesssim&\int_{0}^{\pi}\int_{0}^{\pi}\frac{(w(t)-w(t+s))^2}{s^2} \frac{\rho_{\mu,E}(t)}{t}{N}_E(t)dsdt,
\end{eqnarray*}
and
\begin{eqnarray*}
\cJ_2&=&\frac{1}{2\pi}\int_\TT\int_{\delta\leq \delta'}\frac{(w(\delta)-w(\delta'))^2}{|\zeta-\zeta'|^2} d\mu(\zeta')|d\zeta|\\
&\lesssim& \int_\TT\int_{\TT}\frac{(w(\delta)-w(\delta+|\zeta-\zeta'|))^2}{|\zeta-\zeta'|^2} d\mu(\zeta')|d\zeta|\\
&\lesssim&\int_{\zeta \in \TT}\int_{| \zeta '-\zeta  | \leq \delta}+ \int_{\zeta \in \TT}\int_{ |\zeta '-\zeta | \geq \delta}=\cJ_{21}+\cJ_{22} .
\end{eqnarray*}
Clearly we have
$$\cJ_{21}\lesssim  \int_\TT  w'( \delta)\rho_{\mu,E}(2\delta)  |d\zeta|
\lesssim  \int_0^{\pi}  w'(t )^2\rho_{\mu,E}(2t)  N_E(t)dt.
$$
With the same calculation, as in Lemma \ref{lem11}, we have
$$
\cJ_{22} \lesssim  \int_{t=0}^{\pi}\int_{s=t}^{\pi}\frac{(w(t)-w(t+s))^2}{s^2} \frac{\rho_{\mu,E}(s)}{s}{N}_E(t)dsdt. 
$$
Since $\cD_\mu(w)=\cJ_1+\cJ_2$, we get our result.
\end{proof}

For a positive increasing function $\psi$ such that $\psi (0)=0$, we set \begin{equation}
\label{concave-convexe}
M_{\psi,E}(s)= \max\Big(\int_{0}^{s}\frac{{\psi(t)}} {t}N_E(t)dt,\; \frac{\psi(s)}{s}|E_s|\Big),
\end{equation}
where  $E_t=\{\zeta\in E\text{ : } \dist(\zeta,E)\leq t\}$. If $\psi$ is concave, then $\psi(x)/x$ is decreasing and 
$$\frac{\psi(s)}{s}|E_s|\leq  \frac{\psi(s)}{s}\int_{0}^{s}N_E(t)dt\leq \int_{0}^{s}\frac{{\psi(t)}} {t}N_E(t)dt=M_{\psi,E}(s) .$$
And if $\psi$ is convex then  $\psi(x)/x$ is increasing, so 
$$\int_{0}^{s}\frac{{\psi(t)}} {t}N_E(t)dt\leq \frac{\psi(s)}{s}\int_{0}^{s}N_E(t)dt=\frac{\psi(s)}{s}|E_s|=M_{\psi,E}(s).$$

The function $\psi$ is called $\alpha$--admissible if $\psi$ is concave or convex and $\psi(s)/s^\alpha$ is decreasing for some $\alpha>0$.  
Now we can state the main result of this subsection

\begin{theo}\label{thhar} Let $E$ be a closed subset of $\TT$. Let $w$ be a convex decreasing function and let  $\Omega(\zeta)=w(d(\zeta,E))$. Suppose that there exists an $\alpha$--admissible function $\psi$, with $\alpha<2$, such that  $\rho_{\mu,E}(s)\leq \psi(s)$.  Then 
$$\cD_\mu(\Omega)\leq C ({\alpha })\| w'M_{\psi,E} \|_{\infty} \| w \|_{\infty}.$$
\end{theo}

\begin{proof} 

We apply lemma \ref{lem22}. An analogue  calculation, as in the proof of Theorem \ref{coronorm}, gives 
$$\cI_1+\cI_2\lesssim  \| w'M_{\rho_{\mu,E},E} \|_{\infty} \| w \|_{\infty}.$$
Now we consider the integral $\cI_3$.
We have 
\begin{eqnarray*}
\cI_3&=&2 \int_{t=0}^{\pi}\int_{s=t}^{\pi}\int _{u=t}^{t+s}\int _{v=t}^{t+s}w'(u)w'(v) \frac{\rho_{\mu,E}(s)}{s^3}{N}_E(t)dvdudsdt\\
&\lesssim&\int_{t=0}^{\pi}\int_{s=t}^{\pi}\int _{u=t}^{2s}\int _{v=t}^{2s}|w'(u)||w'(v)| \frac{\rho_{\mu,E}(s)}{s^3}{N}_E(t)dvdudsdt\\
&\lesssim&\int_{v=0}^{\pi} |w'(v)|\frac{\psi(v)}{v^2}\int_{u=0}^{v} |w'(u)| |E_u|dudv\\
&\lesssim&\int_{v=0}^{\pi} |w'(v)|\frac{1}{v^{2-\alpha}}\int_{u=0}^{v} |w'(u)| \frac{\psi(u)}{u^\alpha}|E_u|dudv\\
&\leq &C(\alpha)\sup _u \Big ( |w'(u)|\frac{\psi(u)}{u} |E_u|\Big )\int_{v=0}^{\pi} |w'(v)|\frac{1}{v^{2-\alpha}}v^{2-\alpha}dv\\
&\leq &C(\alpha)\sup _u \Big (| w'(u)|\frac{\psi(u)}{u} |E_u|\Big ) \| w \|_{\infty}.
\end{eqnarray*}
\end{proof}

\begin{coro}Let $E$ be a closed subset of $\TT$. Let $w$ be a convex decreasing function and let  $\Omega(\zeta)=w(d(\zeta,E))$. Suppose that  $\rho_{\mu,E}(t)=O(t^\alpha)$ for some $\alpha >0$.  Then 
\begin{enumerate}
\item  $\displaystyle \cD_\mu(\Omega)\leq C(\alpha)\sup_{t\geq 0}\Big|w'(t)\int_{t}^{\pi} t^{\alpha-1}N_E(t)dt\Big|  \| w \|_{\infty}$,   if $0<\alpha\leq 1$, 
\item $\displaystyle \cD_\mu(\Omega)\leq C(\alpha)\sup_{t\geq 0}\big| w'(t)t^{\alpha-1}|E_t|\big|  \| w \|_{\infty}$,    if $1\leq \alpha< 2$,
\item $\displaystyle \cD_\mu(\Omega)\leq C(\alpha)\sup_{t\geq 0}\big| w'(t) |\log t||E_t|\big|  \| w \|_{\infty}$,   if $\alpha= 2$,
\item   $\displaystyle \cD_\mu(\Omega)\leq C(\alpha,h)\sup_{t\geq 0}\big| w'(t) h(t)|E_t|\big|  \| w \|_{\infty}$,  if $\alpha>2$, where $h$ is a positive increasing function  such that $h(0)=0$ and $\int_0ds/h(s)<\infty$.\end{enumerate}
\end{coro}
\begin{proof} (1) and (2) are direct consequences of Theorem \ref{thhar}. \\
Now we prove (3). The proof of Theorem \ref{thhar} gives
$$\cI_3\lesssim  \int_{v=0}^{\pi} |w'(v)|\log \frac{1}{v}\int_{u=0}^{v} |w'(u)| |E_u|dudv,$$
and we get our estimate.\\
Finally we prove (4).
 Since $|E_t|\to 0$, then there exists a  positive increasing function $h$, $h(0)=0$,  such that $\int_0ds/h(s)<\infty$.  Again, by the proof of Theorem \ref{thhar} we have 
$$\cI_3\lesssim  C(\alpha)\int_{v=0}^{\pi} |w'(v)|\int_{u=0}^{v} \frac{h(u)}{h(u)} |w'(u)| |E_u|dudv, $$
which gives  the desired estimate.
\end{proof}


\section{ kernel estimate}
In this section we will prove  Theorem \ref{thkernel}.  The reproducing kernel $k^\mu$ of $  \cD(\mu)$ is defined by,
$$f(z)=\langle f,k^\mu(\cdot,z)\rangle,\qquad f\in   \cD(\mu),\quad z\in \DD.$$
So,   
    \begin{equation}
  \label{kernelid}
  k^\mu(z,z)=\sup\{|f(z)|^2\text{ : } f\in \cD(\mu),\;  \|f\|_\mu ^2\leq 1\}.
  \end{equation}
It follows obviously that for $|z|\leq1/2$ we have 
$$
  k^\mu(z,z)\asymp 1+ \int_{0}^{|z|} \frac{dr}{(1-r)P[\mu](rz/|z|)+(1-r)^2}.
$$
By Littlewood--Paley identity, we have  
\begin{eqnarray*}
  \|f\|^2_\mu&= &\|f\|^2_{H^2}+ \cD _\mu(f)\\
  &=& |f(0)|^2+\displaystyle \int_\DD|f'(w)|^2|[|\log{|w|}|+P[\mu](w)]dA(w)\\
  &\asymp & |f(0)|^2+\displaystyle \int_\DD|f'(w)|^2|[(1-|w|)+P[\mu](w)]dA(w).
\end{eqnarray*}
Let $f\in \cD(\mu)\backslash\{0\}$ and let $f=I_fO_f$ be the inner--outer factorization of $f$. Then  $O_f\in \cD(\mu)$ and $\cD_\mu(O_f)\leq \cD_\mu(f)$ (see \cite{RS1}). Thus by  \eqref{kernelid}, we get 
 \begin{equation}\label{KernelO}
  {k^\mu(z,z)}=\sup\{|f(z)|^2\text{ : } f\in \cD(\mu) \text{ outer function and} \;  \|f\|_\mu\leq 1\}.
 \end{equation}
This observation will be useful in the proof of the lower estimate. 

 \subsection{Proof of the upper estimate}  
Let $z=\rho \in [1/2, 1)$. By  Lemma \ref{lemFmu1}, it suffices to prove that 
$$k^\mu(\rho,\rho )\lesssim 1+ \int_{1-\rho}^1 \frac{dx}{F_\mu (x)+x^2}.$$
Let $f\in \cD(\mu)$, since  $f\in H^2$, 
  $$|f(iy)|\leq \frac{\|f\|_{H^2}}{\sqrt{1-y}}\leq \sqrt{2}\|f\|_{H^2}, \qquad 0 <y<1/2.$$
  So, for $ 0 <y<1/2$,  we have 
  $$
|  f(\rho+i(1-\rho) y)|=\Big|f(iy)+\int_{0}^{\rho}f'(t+i(1-t)y)dt\Big|
  \lesssim\int_{0}^{\rho}|f'(t+i(1-t)y)|dt+ \|f\|_{H^2}.
  $$
  
  Let  $\Delta$ be the triangle   with vertices $-i/2$, $1$, $i/2$ and let 
   $\Delta_\rho=\{x+iy\in \Delta\text{ : } 0\leq x \leq \rho\}$. 
By change of variables $w=u+iv=t+i(1-t)y$, we get  
\begin{eqnarray}\label{estimation-kernel}
  \frac1{1-\rho}\int_{-(1-\rho)/2}^{(1-\rho)/2}|f(\rho+i\eta)|d\eta&=&
     \int_{-1/2}^{1/2}|f(\rho+i(1-\rho) y)|dy\nonumber  \\
     &\lesssim& \int_{0}^{\rho}\int_{-1/2}^{1/2}|f'(t+i(1-t)y)|dydt+ 
\|f\|_{H^2}\nonumber \\
&\lesssim& \int_{\Delta_\rho}|f'(w)|\frac{dudv}{1-u}+ 
\|f\|_{H^2}\nonumber\\
&\lesssim& \cD_\mu(f)^{1/2}\Big[\int_{\Delta_\rho} \frac{dA(w)}{(1-u)^2((1-|w|)+P[\mu](w))}\Big]^{1/2}+ 
\|f\|_{H^2}\nonumber\\
&\lesssim& \cD_\mu(f)^{1/2}\Big[\int_{0}^{\rho} \frac{du}{(1-u)^2+(1-u)P[\mu](u)}\Big]^{1/2}+ 
\|f\|_{H^2}\nonumber\\
&\lesssim & \cD_\mu(f)^{1/2}\Big[\int_{1-\rho}^{1} \frac{dx}{F_\mu(x)+x^2} \Big]^{1/2}+ 
\|f\|_{H^2}.
\end{eqnarray}
Denote by $D(\lambda,r)$ the disc of radius $r$ centered at $\lambda$.  Since   
$$D(\rho,(1-\rho)/4)\subset \{z=x+iy\text{ : } |x-\rho|\leq (1-\rho)/4 \text{ and }  |y|\leq (1-x)/4\},$$
 by \eqref{estimation-kernel}  and the subharmonicity  of $|f|$ we obtain
  \begin{eqnarray}\label{eqkernel1}
|f(\rho)|
&  \lesssim&\frac{1}{(1-\rho)^2} \int_{x=\rho-\frac{1-\rho}{4}}^{\rho+\frac{1-\rho}{4}}\Big(\int_{y=-\frac{1-x}{2}}^{\frac{1-x}{2}}|f(x+iy) dy\Big)dx\nonumber\\
&\lesssim &\cD_\mu(f)^{1/2} \Big[\int_{\frac{5}{4}(1-\rho)}^1\frac{dx}{F_\mu(x)+x^2}\Big]^{1/2}+\|f\|_{H^2} 
\end{eqnarray}
Now  from \eqref{kernelid}, we get    
$$k^\mu(\rho,\rho )\lesssim1+ \int_{1-\rho}^1 \frac{dx}{F_\mu (x)+x^2} \asymp  1+ \int_{0}^{|z|} \frac{dr}{(1-r)P[\mu](rz/|z|)+(1-r)^2}.$$

\subsection{Lower estimate}

\subsubsection{Weighted Sobolev spaces} 
In this subsection we introduce weighted Sobolev spaces which will be used in the proof of lower estimate of the norm of the kernel of $\cD(\mu)$ and in the proof of Theorem \ref {capdirichlet}. Let $\varphi$ be a nondecreasing continuous function. The Sobolev space associated to $\varphi$, $W^{\infty}(\varphi)$, consists of real continuous functions $f$ on $]0,{2\pi}]$ given by \begin{equation}\label{representation} 
f(x) = f(2\pi)+\displaystyle \int _x^{2\pi} g(s)ds, \qquad 0<x \leq 2\pi,
\end{equation}
where $g$ is a measurable function satisfying 
$$||g\varphi ||_{\infty}=\sup_{0<x} |g(x)|\varphi (x)<\infty$$ 
 As usual $g$ will be denoted by $f'$. Equipped with the following norm 
 $$\| f\|_{W^{\infty}(\varphi )}= \|f'\varphi \|_{\infty},$$
  $W^{\infty}(\varphi)$ is a Banach space. It becomes a topological Banach algebra, if and only if, 
$$
\displaystyle \int _{{0}}\frac{dx}{\varphi(x)}<\infty . 
$$
We say that $f$ is regular, and write $f \in {\cal R}$, if $f$ is a $C^1$ convex decreasing function on $[0,2\pi]$, satisfying $f(2t) \leq 2f(t)$ and $t^2|f'(t)|$ is increasing.\\

Our goal is to estimate 
$$ \gamma_\varphi (a)=  \sup \big\{f^2(a): f\in  W^{\infty}(\varphi)\cap \cR, \; \| f\|_{W^{\infty}(\varphi)}\|f\|_{\infty}+\|f\|_2^2\leq 1\big\}.$$
First we will examine the hilbertian case 
 $$
W^2(\varphi)= \big\{f \ \text{of the form (\ref{representation}})\ : \|f\|_{W^2(\varphi)}^2=\|f\|_2^2+\displaystyle \int _{{0}}^{2\pi} |f'(t)|^2\varphi (t)dt <\infty \big\}.
$$
We also need the following subspace of $W^2(\varphi)$
$$
W_0^2(\varphi)= \big\{f \ \text{of the form (\ref{representation}})\ : f(2\pi ) =0, \ \mbox{and}\ \|f\|_{W_0^2(\varphi)}^2=\displaystyle \int _{{0}}^{2\pi} |f'(t)|^2\varphi (t)dt <\infty \big\}.
$$
Clearly, evaluations at points of $ ]0,2\pi]$ define continuous linear functionals on $W^2(\varphi)$ and on $W_0^2(\varphi)$.  Let $K_\varphi,  L_\varphi$ be the reproducing kernels of $W^2(\varphi)$ and of $W_0^2(\varphi)$ respectively.
One can give, with an elementary calculation,  the expression of the reproducing kernel $L_{\varphi}$  of $W_0^2(\varphi)$. Indeed we have
$$
L_{\varphi}(t,s)=\left\{
\begin{array}{lll}
\displaystyle \int _{{t}}^{2\pi} \frac{dx}{\varphi (x)}\qquad & t\geq s,\\
&\\
L_{\varphi}(s,s)\qquad& t\leq s.\\
\end{array}
\right.
$$
The estimates of $K_\varphi,$ on the diagonal is given by
$$
K_{\varphi}(a,a)\asymp 1+L_\varphi (a,a)= 1+\displaystyle \int _{{a}}^{2\pi} \frac{dx}{\varphi (x)}.
$$
It means that
$$
  \sup \big\{f(a)^2 \text{ : } f \in  W^{2}(\varphi) \text{ , }  \|f\|^2_{W^2(\varphi)}\leq 1\big\}\asymp 1+\displaystyle \int _{{a}}^{2\pi} \frac{dx}{\varphi (x)}.
$$
The following proposition will be used several times in what follows.
\begin{prop}\label{sobolev}
Suppose that $t^2/\varphi(t)$ is increasing and $\varphi (t) \geq t^2$.  Let $a<1/2$, we have
$$
\gamma_\varphi (a) \asymp K_{\varphi}(a,a).
$$
\end{prop}
\begin{proof}
Since $\|f\|_{W^2(\varphi )}^2\leq  \|f\|_{W^{\infty}(\varphi )}\|f\|_{\infty}$, then $\gamma_\varphi(a) \leq K_{\varphi}(a,a)$. 

Conversely, let $f=f_0+1$, where
$$
f_0(x) = \displaystyle \int _{2\pi \frac{x+a}{2\pi+a}}^{2\pi}\frac{ds}{\varphi(s)}, \qquad 0<x\leq 2\pi.
$$
Clearly $f\in \cR $. Since $\varphi(2t)\asymp \varphi(t)$, we have 
\begin{eqnarray*}
\|f_0\|_2^2&= &  \int _0^{2\pi} f_0(t)^2dt\\
&\lesssim &\int _0^\pi  \int _{u=t+a} ^{2\pi} \int _{v=t+a}^ {2\pi}  \frac{1}{\varphi (u)\varphi (v)}dudvdt\\
&\lesssim &\int _{t=0}^a  \int _{u=a} ^{2\pi} \int _{v=a}^ {2\pi}  \frac{1}{\varphi (u)\varphi (v)}dudvdt+ 
\int _{t=a}^\pi  \int _{u=t} ^{2\pi} \int _{v=t}^ {2\pi}  \frac{1}{\varphi (u)\varphi (v)}dudvdt.\\
&\lesssim &\int _{u=a} ^{2\pi} \int _{v=a}^ {2\pi}  \frac{a}{\varphi (u)v^2}dudv+ 
 \int _{u=a} ^{2\pi} \int _{v=u}^ {2\pi}  \frac{u}{\varphi (u)v^2}dudv.\\
& \lesssim & \| f_0\| _{\infty}.\\
\end{eqnarray*}
Then we obtain
$$K_{\varphi}(a,a) \lesssim 
\frac{f(a)^2}{\|f\|_{W^{\infty}(\varphi )} \| f\|_{\infty}+\|f\|_2^2} \leq \gamma_\varphi (a). 
$$
And the proof is complete.
\end{proof}

\subsubsection{Proof of the lower estimates}  

Let $z=r \in ]1/2,1[$ and let $w\in \cR$. We consider the outer function $f_{w}$ given by
$$|f_{w}(e^{it})|=w (|t|), \qquad  \text{  a.e on } [-\pi,\pi].$$ 

  
   Let $\varphi (x)= F_\mu (x)+ x^2$. By \eqref{KernelO}, Proposition \ref{estimationharmo} and  Theorem \ref{coronorm} we have
\begin{eqnarray*}
k^{\mu}(r,r) &\gtrsim &\frac{|f_{w}(r)|^2}{\|f_{w}\|_{\cD (\mu)}^2}\\
&\gtrsim & \frac{w ^2(1-r) }{\|w \| _{W^{\infty }(F_{\mu})}\|w \| _{\infty}+ \|w \|_2^2}\\
& \gtrsim &  \frac{w  ^2(1-r) }{\|w \| _{W^{\infty }(\varphi)}\|w \| _{\infty}+ \|w \|_2^2}.
\end{eqnarray*}
This implies that
$$k^{\mu}(r,r) \gtrsim  \gamma_{\varphi} (1-r).$$
By Proposition \ref {sobolev}, we obtain the result.

\section{Capacity}
Let $\mu$ be a positive finite measure on $\TT$ and let $c_\mu$ be the capacity given by \eqref{cmucap}
 
The capacity $c_\mu$ is a Choquet capacity \cite{Ca,G}  and so for every borelian set of $\TT$ we have
$$c_\mu(E)=\sup\{c_\mu(K)\text{ : }K\text{ compact }, K\subset E\}.$$
Note that $c_\mu$ satisfies the weak-type inequality. Namely: 
$$c_\mu(\{\zeta\in\TT\text{ : } |f(\zeta)|\geq t  \;\;\;\;\text{Ê$c_\mu$-q.e.}\})\leq \frac{\|f\|^{2}_{\mu}}{t^2},\qquad f\in \cD^{h}(\mu).$$ 
As consequence of this inequality we  have  the following properties.
\begin{prop}The following properties are satisfied
\begin{itemize}
\item Let $E\subset \TT$ be a Borel set and let $\cM_\mu(E):=\{f\in \cD(\mu) \text{ : } g|E=0 \text{ \;  } c_\mu\text{-q.e} \}.$ Then  the set $\cM_\mu(E)$ is closed in $\cD(\mu)$.\\
\item If $f\in \cD(\mu)$ is cyclic for $\cD(\mu)$ then $f$ is outer function and $c_\mu({Z_\TT}(f))=0$.\\
\item Every function $f\in   \cD(\mu)$ has non-tangential limits $c_\mu $-q.e on $\TT$, more precisely  
The radial limit $\displaystyle \lim_{r\to 1-}f(r\zeta)$ exists and is finite for every $f\in \cD^h(\mu)$ if and only if 
$c_\mu(\zeta)>0$.
\end{itemize}

\end{prop}
\begin{proof} See \cite{C, G}.
\end{proof}
Now we will give the proof of the estimate of the capacity of arcs.

\subsection*{Proof Theorem \ref{cap arc}}  Suppose that $\zeta=1$. Let $w\in \cR$ and let $f_w$ be the outer function satisfying 
$$|f_{w}(e^{it})|=w (|t|), \qquad  \text{  a.e on } [-\pi,\pi].$$  
It's clear that $w(|I|)\lesssim w(x)$ for $|x|\leq  2|I|$. 
We have
$$\gamma _\varphi (|I|)\lesssim 1/c_\mu(I ).$$
By proposition \ref {sobolev}, we obtain 
$$c_\mu (I)\lesssim 1/k^\mu(\rho, \rho),\ \ (\rho = 1-|I|).$$

For the reverse inequality note that 
$$c_\mu(I)=\inf\{\|f\|^2_\mu\text{ : } f\in C^1, 0\leq f\leq 1\text{ and } f=1 \text{ on } I\}.$$ 
Consider the function $u\in \cD^h(\mu)$ such that $0\leq u\leq 1$ and $u_{|I}=1$. Hence $u\in \cD^h(\mu)\cap C^1$. We have $P[u](1-|I|)\asymp 1$. Let $\rho=1-|I|$, a similar argument, as in the proof of  \eqref{eqkernel1}, gives 
$$P[u](\rho)\lesssim \cD_\mu(u)^{1/2} \sqrt{k^{\mu}(\rho,\rho)} +\|u\|_{L^2(\TT)}.$$
So ${k^{\mu}(\rho,\rho)^{-1/2}}\lesssim  \cD_\mu(u)^{1/2}$, and 
$$c_{\mu}(I)\geq \frac{1}{k^{\mu}(\rho,\rho)}.$$ 
\hfill $\Box$

As an immediate consequence, we obtain 
\begin{coro}\label{cap1} Let $\lambda\in \TT$.
$$c_\mu(\{\lambda\})=0\iff\int_0^1\frac{dx}{(1-x) P[\mu](x\lambda)+(1-x)^2}=\infty.$$
\end{coro}

Now we will give some sufficient conditions on a closed subset of $\TT$ to be $\mu$--polar. 
Let $E$ be a closed subset of $\TT$.  We define $n_E(\varepsilon)$, the {\it $\varepsilon$--coverning number of $E$},  to be the smallest  number of the closed arcs of length $2\varepsilon$ that cover $E$.   Note that 
$$\varepsilon n_E(\varepsilon)\leq |E_\varepsilon|\leq 4\varepsilon n_E(\varepsilon),\qquad 0<\varepsilon\leq \pi.$$

Let 
 $$\kappa_\mu(r)=\inf\{ k^{\mu}(r\zeta,r\zeta)\text{ : } \zeta\in \supp \mu\}.$$
   It 's easy to see that $\kappa_\mu$ is unbounded  if and only if,  for each $\zeta\in \TT$, we have $c_\mu(\zeta)= 0$. In this case 
one can prove easily, by the sub-additivity property of  capacity, that 

\begin{coro}
Let $E$ be a closed subset of $\TT$ such that $ \displaystyle \lim _{r \to 1^-}\kappa_\mu(r) = \infty$. If 
$$n_E(\varepsilon)=o(\kappa_\mu(1-\varepsilon)),\qquad \varepsilon\to 0,$$
 then $c_\mu(E)=0$.
\end{coro}
Recall that for a positive increasing function $\psi$ on $(0,2\pi)$ such that $\psi (0)=0$, we set 
 $$
\label{concave-convexe}
M_{\psi,E}(s)= \max\Big(\int_{0}^{s}\frac{{\psi(t)}} {t}N_E(t)dt,\; \frac{\psi(s)}{s}|E_s|\Big), \qquad s\in (0,2\pi)
$$
 Now we can state the main result of this section.
\begin{theo}\label{capdirichlet} Let $E$ be closed subset of $\TT$ such that $\rho_{\mu,E}\leq \psi$ where  $\psi$ is $\alpha$--admissible  for some $\alpha < 2$.  If
$$\int_{0}^{\pi}\frac{dt}{M_{\psi,E}(t)}=+\infty,$$
then $c_\mu(E)=0$.
\end{theo}
\begin{proof}
Note that   $ M_{{\psi},E}$  is given by \eqref{concave-convexe}, so 
$M_{\psi,E}(s)$ is increasing. 
 Let $a>0$. By the definition of capacity and using Theorem \ref{thhar}   and  Proposition \ref{sobolev} we have 
$$
 K_{M_{\psi,E}} (a,a)\asymp \gamma^2 _{M_{\psi,E}}(a)\lesssim \frac{1}{c_\mu (E)}.
$$
When $a$ goes to zero, we get
$$\int_{0}^{\pi}\frac{dt}{M_{\psi,E}(t)} \lesssim \frac{1}{c_\mu (E)}.$$
And the proof is complete.
\end{proof}

 \subsection*{Remarks} Now we give some examples: \\
 
{\bf(i)} Let $E $ be a closed subset of $\TT$ and let $\widetilde{M}_{\mu, E}(t)= 
\int _0^t ({\rho _{\mu}(s)}N_E(s)/s)ds$.
 If
$$\int_{0}^{\pi}\frac{dt}{\widetilde{M}_{\mu, E}(t)}= \infty$$
then $c_\mu(E)=0$.\\

{\bf(ii)} If $\mu =m$ is the Lebesgue measure, then $c_\mu $ is comparable to  the logarithmic capacity and $\widetilde{M}_{\mu, E}(s) = |E_s|$. And theorem  \ref {capdirichlet} says that if $\displaystyle \int _0dt/|E_t| $ diverges then $c(E)=0$. This result is du to Carleson \cite[Theorem 2, p.30]{Ca}.\\

{\bf(iii)} Let $K$ be a closed subset of $\TT$ such that $\rho_{\mu,K}(s)= O(s^{1+\beta})$ for some $0<\beta $. If $\beta <1$, then every subset of $K$ with Hausdorff dimension less than $\beta$ is $\mu$-polar. If $\beta >1$, then every subset $E$ of $K$ with $|E|=0$ is $\mu$-polar.\\ 
The following measures $d\mu (\zeta)= d(\zeta, K)^\beta dm(\zeta)$ provide such examples.


\begin{thebibliography}{99}
\bibitem{AH} D. Adams, L. Hedberg, Function spaces and potential theory, Springer, Berlin 1996. 
\bibitem{AM}J. Agler, J. McCarthy,  Pick Interpolation and Hilbert Function Spaces. Graduate Studies in Mathematics, vol. 44. Providence, RI: American Mathematical Society, 2002

\bibitem{BD} A. Beurling, J. Deny, Dirichlet spaces, Proc. Nat. Acad. Sci. USA 45 (1959) 208--215.

\bibitem{Ca}L. Carleson, Selected Problems on Exceptional sets. Van Nostrand Mathematical Studies, No. 13. Princeton, NJ: Van Nostrand.  
\bibitem{C}R. Chac\'on, Carleson-type measure on  Dirichlet type  spaces. Proc. Amer. Math. Soc.  139 (2011) 1605-1615. 

\bibitem{EKMR} O. El-Fallah, K. Kellay, J. Mashreghl, T. Ransford. A primer on the Dirichlet space. Cambridge Tracts in Mathematics. Cambridge University Press, Cambridge, 2014.

\bibitem{EKR}{O.El-Fallah, K.Kellay, T.Ransford, On the Brown-Shields conjecture for the cyclicity in the Dirichlet space. Adv.  Math. 222 (2009) 2196--2214. }

\bibitem{FOT} M. Fukushima, Y. Oshima, M. Takeda, Dirichlet Forms and Symmetric Markov Processes, extended edition, de Gruyter Stud. Math., vol. 19, Walter de Gruyter \& Co., Berlin, 2011
\bibitem{G}D. Guillot, Fine boundary behavior and invariant subspaces of harmonically weighted Dirichlet spaces. Complex Anal. Oper. Theory. 6 (2012) 1211--1230

\bibitem{M} 
N. Meyers, A theory of capacities for potentials of functions inLebesgue classes, Math. Scand., 26 (1970), 255--292
\bibitem{Ri}
 S. Richter,  A representation theorem for cyclic analytic two-isometries, Trans. Am. Math. Soc., 328 (1994) 1--26.
 
\bibitem{Ri1}
 S. Richter, Invariant subspaces of the Dirichlet shift, J. Reine Angew. Math., 386 (1988) 205-220. 
 
\bibitem{RS1}
 S. Richter and C. Sundberg, A formula for the local Dirichlet integral,
  Michigan Math. J., 38  (1991) 355-379. 
  
 \bibitem{Seip} K.~Seip, {\it Interpolating and Sampling in Spaces of Analytic 
Functions}, American Mathematical Society, Providence, 2004.

 \bibitem{S1}S. Shimorin,  Complete Nevanlinna--Pick property of Dirichlet type  spaces. J. Funct. Anal., 191 (2002) (2)  276--296.
\end{thebibliography}
\end{document}